\documentclass[reqno,15pt]{amsart}
\usepackage{amsmath}
\usepackage{amssymb}
\usepackage{amscd}
\usepackage{graphics}
\usepackage{latexsym}
\usepackage{hyperref}

\theoremstyle{plain}
\newtheorem{theorem}{Theorem}[section]
\newtheorem{thm}[theorem]{Theorem}
\newtheorem{cor}[theorem]{Corollary}
\newtheorem{lem}[theorem]{Lemma}
\newtheorem{prop}[theorem]{Proposition}

\newcounter{kludge}
\newcounter{kludgeb}
\theoremstyle{definition}
\newtheorem{defn}[theorem]{Definition}

\newtheorem{rmk}[theorem]{Remark}

\newtheorem{notat}[theorem]{Notation}

\newtheorem{hyp}[theorem]{Hypothesis}

\theoremstyle{remark}

\input xy
\xyoption{all}

\newcommand{\marpar}[1]{}
\newcommand{\mni}{\medskip\noindent}

\newcommand{\mbb}{\mathbb}

\newcommand{\ZZ}{\mbb{Z}}

\newcommand{\RR}{\mbb{R}}

\newcommand{\PP}{\mbb{P}}

\newcommand{\mc}{\mathcal}

\newcommand{\mcC}{\mc{C}}

\newcommand{\mcM}{\mc{M}}

\newcommand{\mcT}{\mc{T}}

\newcommand{\OO}{\mc{O}}

\newcommand{\ol}{\overline}
\newcommand{\ul}{\underline}

\newcommand{\SP}{\text{Spec }}
\newcommand{\Hilb}{\text{Hilb}}
\newcommand{\HB}[2]{\Hilb^{#1}_{#2}}

\newcommand{\Gm}[1]{\mathbb{G}_{#1}}

\newcommand{\Sec}{\text{Sec}}

\newcommand{\pr}[1]{\text{pr}_{#1}}

\newcommand{\Xx}{X}
\newcommand{\Ll}{\mathcal{L}}
\newcommand{\Cc}{C}

\newcommand{\kp}{\kappa}
\newcommand{\Rr}{R}

\newcommand{\FR}{K}
\newcommand{\Lb}{\Lambda}

\newcommand{\cb}{b}

\newsavebox{\sembox}
\newlength{\semwidth}
\newlength{\boxwidth}

\newsavebox{\semrbox}
\newlength{\semrwidth}
\newlength{\boxrwidth}

\title
{Rational points of rationally simply connected varieties over global
  function fields} 

\author[Starr]{Jason Michael Starr}
\address{Department of Mathematics \\
  Stony Brook University \\ Stony Brook, NY 11794}
\email{jstarr@math.stonybrook.edu} 

\author[Xu]{Chenyang Xu}
\address{BICMR \\
Beijing, China}
\email{cyxu@math.pku.edu.cn}

\date{\today}

\begin{document}

%%%%%%%%%%%%%%%%%%%%%%%%%%%%%%%%%%%%%%%%%%%%%%%%%%%%%%%%%%%%%%%%%%%%
%%
%% Abstract
%%
%%%%%%%%%%%%%%%%%%%%%%%%%%%%%%%%%%%%%%%%%%%%%%%%%%%%%%%%%%%%%%%%%%%%

\begin{abstract}
  A complex projective manifold is rationally connected,
  resp. rationally simply connected, if 
  finite subsets are connected by a rational curve, resp. 
the spaces parameterizing these
  connecting rational curves are themselves rationally connected.  We
  prove that a projective scheme over a global function field with
  vanishing ``elementary obstruction''
  has a rational point if it deforms to a rationally simply connected
  variety in characteristic 0.  This gives new, uniform
  proofs over these fields of the Period-Index Theorem, the
  quasi-split case of Serre's ``Conjecture II'', and Lang's $C_2$ property. 
%as
%  well as proving height bounds that are independent of the characteristic.
\end{abstract}

%%%%%%%%%%%%%%%%%%%%%%%%%%%%%%%%%%%%%%%%%%%%%%%%%%%%%%%%%%%%%%%%%%%%%%
%%
%% Body
%%
%%%%%%%%%%%%%%%%%%%%%%%%%%%%%%%%%%%%%%%%%%%%%%%%%%%%%%%%%%%%%%%%%%%%%%

\maketitle

%% \tableofcontents

%%%%%%%%%%%%%%%%%%%%%%%%%%%%%%%%%%%%%%%%%%%%%%%%%%%%%%%%%%%%%%%
%%
%% Section: Statement of the Theorem
%% 
%%%%%%%%%%%%%%%%%%%%%%%%%%%%%%%%%%%%%%%%%%%%%%%%%%%%%%%%%%%%%%%

\section{Statement of the Theorem} 
\label{sec-state} \marpar{sec-state}

\mni
For every projective variety over every field, 
the elementary obstruction of Colliot-Th\'{e}l\`{e}ne and Sansuc is an
obstruction to the existence of a rational point, \cite{CTSansuc}.  We
prove that for a global function field, for a projective variety that
is a specialization from characteristic $0$ of a \emph{rationally
  simply connected} variety, \cite{dJHS}, \cite{Zhu}, the elementary
obstruction is the only obstruction to existence of a rational point.
This gives a uniform proof of three classical existence theorems for rational
points of varieties over global function fields: the Period-Index
Theorem, the quasi-split case of Serre's ``Conjecture II'', and Lang's
proof that global function fields have the $C_2$ property.
Using work of Robert
Findley, \cite{Findley}, we give a new existence theorem for rational
points of hypersurfaces in
(twists of) Grassmannians over global function fields.

\mni
In addition to rational simple
connectedness, the key ingredient is the theory of H\'{e}l\`{e}ne
Esnault on rational points over finite fields of varieties with
coniveau $\geq~1$, \cite{Esnaultpadic}.  Although the elementary
obstruction is in general quite subtle, for rationally simply
connected varieties, or more generally those varieties satisfying
hypotheses identified by Yi Zhu, the elementary obstruction has a
straightforward, geometric meaning.  

\begin{defn}\cite[Definition 2.10]{Zhu} \label{defn-Zhu} \marpar{defn-Zhu}
For every projective morphism $f_T:X_T\to T$ of Noetherian schemes,
the \emph{acyclic open subscheme} $T^o$ of $T$ is the maximal open
such that the pullback $X^o = X_T\times_T T^o$ and
$f^o:X^o\to T^o$ satisfy the following,
\begin{enumerate}
\item[(i)] $f^o$ is flat with integral geometric fibers $X_t$,
\item[(ii)] every $X_t$ is locally complete intersection,
  and $X_t$ is smooth in codimension $\leq 3$,
\item[(iii)] $h^1(X_t,\OO_{X_t})$ and $h^2(X_t,\OO_{X_t})$ equal $0$, and
\item[(iv)] $X_t$ is algebraically simply connected. 
\end{enumerate}
If $T^o$ equals $T$, resp. if $T^o$ is dense in $T$, 
then $\Xx_T/T$ satisfies the \emph{acyclic hypothesis}, resp. satisfies
the \emph{acyclic hypothesis generically}.
\end{defn}

\mni
The relative Picard functor
$\text{Pic}_{X^o/T^o}$ 
of $f^o$ is
representable by an \'{e}tale group scheme over $T^o$, 
%uniquely
%determined on all of $T^o$ by its restriction to any dense open,
and it is \'{e}tale
locally a constant sheaf $\ZZ^\rho$.  
The dual \'{e}tale sheaf,
$$
\text{Hom}_{T^o-\text{gr}}(\text{Pic}_{X^o/T^o},\Gm{m,T^o})
$$
is representable by a smooth group scheme $Q$ over $T^o$
that 
is \'{e}tale locally isomorphic to $\Gm{m,T^o}^\rho$.
By construction, the \'{e}tale sheaf
$R^1f^o_*((f^o)^*Q)$ on $U$ equals
$$\text{Hom}_{T^o-\text{gr}}(\text{Pic}_{X^o/T^o},\text{Pic}_{X^o/T^o}).$$

\begin{defn} \label{defn-elem} \marpar{defn-elem}
Assume that the acyclic hypothesis holds generically.
For every $T^o$-scheme $V$, a $Q$-torsor $\mcT_V$ on $X_V=X_T\times_T
V$ 
is \emph{universal} if the associated
global section of $R^1(f_V)_*(f_V^*Q)$ equals the identity morphism of
$\text{Pic}_{X_V/V}$.  
The morphism $f_V$ has
\emph{trivial elementary obstruction} if there exists a
universal $Q$-torsor $\mcT_V$ on $X_V$.  
\end{defn}

\begin{rmk} \label{rmk-elem} \marpar{rmk-elem}
Let $T$ be an integral scheme, and let $f_T:\Xx_T\to T$ be a morphism
that satisfies the  acyclic hypothesis generically.  There exists a rational
section of $f_T$ only if there exists a universal torsor over a dense
open subscheme of $T$, \cite[Proposition 2.11]{Zhu}.
\end{rmk}

\mni
Assuming the acyclic hypothesis, there exists a rational section
only if the elementary obstruction vanishes over some dense open
$V\subset T^o$, cf. \cite[Proposition 2.11]{Zhu}.   
The acyclic hypothesis holds often. In particular,
it holds if the geometric generic fibers of $f_T$ are
smooth and satisfy one of Cases 1 to 4.
\begin{enumerate}
\item[Case 1.]  The fibers are complete
  intersections in $\mathbb{P}^n$ of $\cb$ hypersurfaces of degrees
  $(d_1,\dots,d_\cb)$, $d_i\geq 2$, satisfying $d_1^2 + \dots + d_\cb^2 \leq n$.
\item[Case 2.]  The fibers are
  Grassmannians parameterizing $\ell$-dimensional subspaces of an
  $m$-dimensional vector space.
\item[Case 3.]  The fibers are
  homogeneous for a smooth action of a linear algebraic group.
\item[Case 4.]  The fibers are
  hypersurfaces in a Grassmannian, as in Case 2, whose Pl\"{u}cker
  degree $d$ satisfies $(3\ell-1)d^2 - d < m-4\ell-1$.  
\end{enumerate}

\mni
Let $F(\eta)$ be the function field of a smooth projective curve over
a finite field $F$, i.e., $F(\eta)$ is a \emph{global function field}.  
Let $f_\eta:\Xx_\eta \to \SP F(\eta)$ be a
projective 
morphism. Let $\Rr$ be a Henselian DVR with residue field $F$, and with
fraction field $\FR$ a characteristic $0$ field.

\begin{defn} \label{defn-lift-b} \marpar{defn-lift-b}
An \emph{integral extension over $\Rr$} is a pair of projective
morphisms satisfying the following hypotheses,
$$
(\pi:\Cc_\Rr\to \SP
\RR, f_\Rr :\Xx_\Rr \to \Cc_\Rr).$$
\begin{enumerate}
\item[(1)]
The morphism $\pi$ is flat, has connected geometric fibers, and the $R$-smooth
locus $\Cc_{\Rr,\text{sm}}$ contains both
the generic fiber $\Cc_\FR$ and a generic point $\eta$ of $\Cc_F$
whose residue field equals $F(\eta)$.
\item[(2)]
The flat locus in $\Cc_\Rr$ of the morphism 
$f_\Rr$ 
contains $\eta$, and the fiber of $f_\Rr$ over $\eta$ equals
$f_\eta$.     
\end{enumerate}
\end{defn}

\begin{thm} \label{thm-app} \marpar{thm-app}
For every global function field $F(\eta)$, for every projective
$F(\eta)$-scheme $\Xx_\eta$ 
that is Case 1 to 4 and has a universal torsor, there exists an
integral extension over $\Rr$ such that $\Xx_\FR\to \Cc_\FR$ satisfies
the acyclic hypothesis and has a universal torsor.  There exists an
$F(\eta)$-point of $\Xx_\eta$.
\end{thm}

\mni
The proof is at the end of Section \ref{sec-extend}.  In forthcoming
work we address height bounds and variants of this theorem.

\begin{rmk} \label{rmk-hist} \marpar{rmk-hist}
There are several earlier theorems relating to these four cases.
\begin{enumerate}
\item[Case 1.] 
For complete intersections, Serge Lang proved
existence of $F(\eta)$-points,
cf. \cite[Corollary p. 378]{Lang52}.  Lang's Theorem is sharp: for
every $(n,d_1,\dots,d_\cb)$ with 
$d_1^2+\dots+d_\cb^2=n+1$, there are varieties with no $F(\eta)$-point.
\item[Case 2.] 
If $X_{F(\eta)}$ is geometrically isomorphic to a
Grassmannian and has vanishing elementary obstruction, then
existence of an $F(\eta)$-point follows from the
Brauer-Hasse-Noether theorem.  Conversely, existence of
$F(\eta)$-points
implies the
``Period-Index Theorem'' for $F(\eta)$.
\item[Case 3.]
For every connected, simply connected,
semisimple group $G$ over $F(\eta)$, for every (standard) parabolic
subgroup $P$ of $G$,
for every $G$-torsor $E$,
the variety $X=E/P$ has vanishing
elementary obstruction.  By Harder's proof
of Serre's ``Conjecture II'', $X$ has an $F(\eta)$-point,
cf. \cite{Harder}.  
Conversely,
existence of $F(\eta)$-points  
implies the \emph{split case} of Serre's
``Conjecture II'' for $F(\eta)$, cf. \cite[Section 16]{dJHS}, and even 
the \emph{quasi-split case}, 
cf. \cite{Zhu}.   
%The theorem above for Case 3 depends
%crucially on Yi Zhu's work.
\item[Case 4.]
The theorem above for Case 4 depends on Robert Findley's
work, \cite{Findley}.  There are counterexamples with $m$ equal to
$\ell d^2$  showing that the inequality is
asymptotically sharp up to a linear factor of $3$. 
\end{enumerate}
\end{rmk}

% \begin{rmk} \label{rmk-app} \marpar{rmk-app}
% In Cases 2 to 3, $\Xx_\eta$ is a smooth projective $F(\eta)$-scheme that
% satisfies acyclic hypothesis.  In Cases 1 and 4, $\Xx_\eta$ is a
% complete intersection, resp. hypersurface, in a smooth projective
% $F(\eta)$-scheme that is \'{e}tale locally a projective space,
% resp. Grassmannian.  The hypothesis is that this smooth ambient
% $F(\eta)$-scheme has a universal torsor.
% \end{rmk}

%%%%%%%%%%%%%%%%%%%%%%%%%%%%%%%%%%%%%%%%%%%%%%%%%%%%%%%%%%%%%%%
%%
%% Section: Abel Sequences and Rational Points
%% 
%%%%%%%%%%%%%%%%%%%%%%%%%%%%%%%%%%%%%%%%%%%%%%%%%%%%%%%%%%%%%%%

\section{Abel Sequences and Rational Points} 
\label{sec-Abel} \marpar{sec-Abel}

\mni
The key technique of \cite{dJHS} and \cite{Zhu} is an analysis of the
parameter schemes of rational sections of a fibration over a curve.  
This, in turn, reduces to an analysis of the spaces of rational curves
in the geometric generic fiber of the fibration.  The output of this
analysis is a sequence of parameter spaces of rational sections that
have rationally connected fibers relative to the Abel maps.  This is
an \emph{Abel sequence}.  The Abel maps are defined in terms of
classifying stacks of the Picard torus and their associated coarse
moduli spaces.

\begin{defn} \label{defn-Sec-b} \marpar{defn-Sec-b}
For every integral extension over $\Rr$,
denote by $\Sec(\Xx_\FR/\Cc_\FR/\FR)$ the
$\FR$-scheme parameterizing families of 
sections $\sigma$ of $f_\FR:\Xx_\FR\to \Cc_\FR$,
cf. \cite[Part IV.4.c,
p. 221-219]{FGA}.
Denote by
$\Sigma(\Xx_\FR/\Cc_\FR/\FR)$ the closure of the locally closed immersion
$\Sec(\Xx_\FR/\Cc_\FR/\FR) \to \HB{}{\Xx_\FR/\FR}$ that sends $\sigma$ to the
closed image of $\sigma$.
\end{defn}

\begin{defn} \label{defn-defo} \marpar{defn-defo}
For every integral extension over $\Rr$, for every point $[\mcC_F]$ of
$\Sigma(\Xx_\FR/\Cc_\FR/\FR)(F)$, 
a \emph{deformation after base change} of $[\mc{C}_{F}]$
is a pair $(\Rr\to \Rr',\mc{C}_{\Rr'})$ of a
finite type and local homomorphism of DVRs, ${\Rr}\to \Rr'$, whose field
extensions $K\subset K'$ and $F\subset F'$ are (automatically)
finite, together with 
an ${\Rr}'$-flat closed subscheme $\mc{C}_{\Rr'}$ of $\Xx_{\Rr'}$ that is the
closure of $\sigma_{K'}(\Cc_{K'})$ for a section $\sigma_{K'}$,
and whose closed fiber $\mc{C}_{\kp'}$ is the
base change
$\mc{C}_{\kp} \otimes_{\kp} \kp'$.  
\end{defn}

\begin{notat} \label{notat-V} \marpar{notat-V}
For every integral extension over $\Rr$, for every
$F$-point $[\mc{C}_F]$ of $\Sigma(\Xx_{\Rr}/\Cc_{\Rr}/{\Rr})$,
consider the
restriction $f|_{\mc{C}_F}:\mc{C}_F\to \Cc_F$.  Denote by $V_{\mc{C}_F}$ 
the maximal
open subscheme $V$ of $\Cc_F$ such that $f:\mc{C}_F\cap f^{-1}(V)\to V$ is an
isomorphism.   
\end{notat}

\begin{lem} \label{lem-Sigma} \marpar{lem-Sigma}
For every integral extension over $\Rr$,
every $F$-point $[\mc{C}_F]$ of $\Sigma(\Xx_\Rr/\Cc_\Rr/\Rr)$ has a
deformation after base change.  Moreover,
the open
subscheme $V_{\mc{C}_F}$ of $\Cc_F$ contains every generic point of 
$\Cc_{F,\text{sm}}$, i.e., there exists an open subscheme $V_{\mc{C}_F}$
containing every generic point of $\Cc_{F,\text{sm}}$ such that the
restriction $f:\mc{C}_F \cap f^{-1}(V_{\mc{C}_F}) \to V_{\mc{C}_F}$ is
an isomorphism.   
\end{lem}

\begin{proof}
Since $\Sigma(\Xx_\Rr/\Cc_\Rr/\Rr)$ is the closure of
$\text{Sec}(\Xx_\Rr/\Cc_\Rr/\Rr)$, every $F$-point admits a lift after
base change to an $R'$-point whose generic fiber is in
$\text{Sec}(\Xx_\Rr/\Cc_\RR/\Rr)$.  

\mni
The maximal open subscheme of $\Cc_{\Rr'}$ on which $\sigma_{K'}$ is
regular exactly equals the maximal open subscheme $V_{\mc{C}_{\Rr'}}$ 
on which the
restriction $f:\mc{C}_{\Rr'} \cap f^{-1}(V_{\mc{C}_{\Rr'}}) \to
V_{\mc{C}_{\Rr'}}$ is an isomorphism. 
Since $f$ is proper, by the valuative criterion of
properness, the open
$V_{\mc{C}_{\Rr'}}$ 
contains every codimension $1$ point of
$\Cc_{\Rr'}$ at which $\Cc_{\Rr'}$ is normal.  In particular, it contains
every codimension $1$ point of the base change
$\Cc_{\Rr,\text{sm}}\otimes_{\Rr} \Rr'$, since this is even regular.  In
particular, $\mc{C}_{F'} \to \Cc_{F'}$ is an isomorphism over every generic
point contained in $\Cc_{F,\text{sm}}\otimes_F F'$.  But this is a
property that is preserved by flat base change.  Therefore also $\mc{C}_F
\to \Cc_F$ is an isomorphism over every generic point contained in
$\Cc_{F,\text{sm}}$, i.e., the open subscheme $V_{\mc{C}_F}$ of $\Cc_F$ 
contains every generic point of
$\Cc_{F,\text{sm}}$.  
\end{proof}

\begin{hyp} \label{hyp-lift-b} \marpar{hyp-lift-b}
Let $\Cc_\FR$ be a smooth, projective, geometrically connected
$\FR$-curve.  Let $f_\FR:\Xx_\FR\to \Cc_\FR$ be a projective morphism
that satisfies the acyclic hypothesis.  
\end{hyp}

\mni
Assuming Hypothesis \ref{hyp-lift-b},
%for the associated Picard torus
%$Q$ on $\Cc_\FR$, 
denote by $BQ$ the classifying stack over
$\Cc_\FR$.  Denote by $BQ_{\Cc_\FR/\FR}$ the Hom stack,
$\text{Hom}_{\FR}(\Cc_\FR,BQ)$.  This is a gerbe over
the coarse moduli space $|BQ_{\Cc_\FR/\FR}|$.  The coarse moduli space
is a smooth group $\FR$-scheme.  

\mni
Denote by
$\Lb(Q/\Cc_\FR/\FR)$ the \'{e}tale group $\FR$-scheme that is the 
 pushforward from $\Cc_K$ 
of the
cocharacter lattice of $Q$.  
For every field extension $L/K$, the finitely generated
Abelian group
$\Lb(Q/\Cc_K/K)(L)$ is
the group completion of 
the monoid of
$\text{Gal}(L(\Cc_L))$-invariant 
elements of the Mori cone of the geometric
generic fiber of $f_L$.

\begin{hyp} \label{hyp-positiveMori} \marpar{hyp-positiveMori}
Assume \ref{hyp-lift-b} and
that the geometric generic of $f_\FR:\Xx_\FR\to \Cc_\FR$ has
simplicial Mori cone.
\end{hyp}

\begin{defn} \label{defn-positiveMori} \marpar{defn-positiveMori}
Under Hypothesis \ref{hyp-positiveMori},
denote by $\Lb^+(Q/\Cc_K/K)$ the finitely generated, saturated, 
positive structure on
$\Lb(Q/\Cc_K/K)$ that is the Galois-invariant part of the Mori cone.
Denote by $\theta \in \Lb(Q/\Cc_K/K)(K)$ the
Galois-invariant element that is the sum
of the simplicial
generators of the Mori cone. 
For every $K$-point $e$
of $\Lb(Q/\Cc_K/K)$, the translate by $e$ of $\Lb^+(Q/\Cc_\FR/\FR)$ is
$\Lb^{\geq e}(Q/\Cc_\FR/\FR)$.
Denote by $|e|'$ the smallest non-negative
integer such that $|e|'\theta -e$ is an element of
$\Lb^+(Q/\Cc_K/K)(K)$.  
%Expressing $e$ (uniquely) as a linear
%combination of the
%simplicial generators of the Mori cone with coefficients $c_i\in \ZZ$,
%$|e|'$ equals $\max(0,\{c_i\}_{i})$. 
\end{defn}

\mni
There is a smooth, surjective morphism of smooth group $\FR$-schemes,
$$
\text{deg}_{Q/\Cc_\FR/\FR}:|BQ_{\Cc_\FR/\FR}| \to \Lb(Q/\Cc_\FR/\FR),
$$
whose kernel $|BQ^0_{\Cc_\FR/\FR}|$ is an extension of a commutative,
finite, \'{e}tale group 
scheme by an Abelian variety.  

\begin{defn} \label{defn-WC} \marpar{defn-WC}
The fiber $|BQ^\theta_{\Cc_K/K}|$ is a torsor over $\SP K$ for the
group $K$-scheme $|BQ^0_{\Cc_K/K}|$, 
%i.e., it is an element of the
%Weil-Chatelet group of $|BQ^0_{\Cc_K/K}|$, 
cf. \cite{LangTate}.  
Denote by $m=m(\Xx_K/\Cc_K/K)$ the 
\emph{period} of
$|BQ^\theta_{\Cc_K/K}|$ in the Weil-Chatelet group of the Abelian
$K$-variety $|BQ^\theta_{\Cc_K/K}|$. 
%Denote by
%$\iota=\iota(Q/\Cc_K/K)$ the \emph{index} of $|BQ^\theta_{\Cc_K/K}|$.  
\end{defn}

\mni
%The period is the minimal positive
%integer $m$ (also the greatest common divisor of all such positive integers)
%for which $|BQ^{m\theta}_{\Cc_K/K}|$ has a $K$-point.
If $\Cc_K$ is a geometrically integral curve of genus $g(\Cc_K)>1$,
resp. $g(\Cc_K)=0$, then the period divides
$m_0(2g(\Cc_K)-2)$, resp. $2m_0$, 
where $m_0\ |\ \rho!$ is the order of
the automorphism group of the Mori cone of the geometric generic fiber
of $f_\FR$. 

\begin{notat} \label{notat-ebound} \marpar{notat-ebound}
For every $K$-point $e$ of $\Lb(Q/\Cc_K/K)$, denote by $|e|$ the
least non-negative integer such that $m|e| \geq |e|'$. 
Then $m|e|$ is the least non-negative integer $\ell$ such that
both $\ell\theta-e$ is a $K$-point of $\Lb^+(Q/\Cc_K/K)$ and
$|BQ^{\ell\theta}_{\Cc_K/K}|$ has a $K$-point.
\end{notat}

\begin{hyp} \label{hyp-Sec} \marpar{hyp-Sec}
Assume \ref{hyp-lift-b} and existence of a 
universal torsor $\mcT_\FR$ on $\Xx_\FR$.
\end{hyp}

\begin{notat} \label{notat-Sec} \marpar{notat-Sec} 
%Denote by $\Sec(\Xx_K/\Cc_K/K)$ the
%$K$-scheme parameterizing families of 
%sections $\sigma$ of $f_K:\Xx_K\to \Cc_K$,
%cf. \cite[Part IV.4.c,
%p. 221-219]{FGA}.
For every universal $Q$-torsor $\mcT_\FR$ on $\Xx_\FR$,
denote by 
$$
\alpha_{\mcT_K}:\Sec(\Xx_K/\Cc_K/K) \to |BQ_{\Cc_K/K}|
$$  
the \emph{Abel map},
$[\sigma] \mapsto [\sigma^*\mcT_K]$.  Denote by $\text{deg}_{\mcT_K}$
the composition $\text{deg}_{Q/\Cc_K/K}\circ \alpha_{\mcT_K}$; this is
the \emph{degree map}.
%Denote by
%$\Sigma(\Xx_K/\Cc_K/K)$ the closure of the locally closed immersion
%$\Sec(\Xx_K/\Cc_K/K) \to \HB{}{\Xx_K/K}$ that sends $\sigma$ to the
%closed image of $\sigma$.  
For every point $K$-point $e$ of
$\Lb(Q/\Cc_\FR/\FR)$, denote by $\Sec^e(\Xx_K/\Cc_K/K)$ the fiber over
$e$ of $\text{deg}_{\mcT_\FR}$.  Denote by
$\Sigma^e(\Xx_\FR/\Cc_\FR/\FR)$ the closure of
$\Sec^e(\Xx_\FR/\Cc_\FR/\FR)$ in $\Sigma(\Xx_\FR/\Cc_\FR/\FR)$.      
\end{notat}

\begin{defn} \label{defn-Abelseq} \marpar{defn-Abelseq}
Under Hypotheses \ref{hyp-positiveMori} and \ref{hyp-Sec},
a \emph{pseudo Abel sequence}, resp. an \emph{Abel sequence}, 
for $(\Xx_K/\Cc_K/K,\mcT_K)$ is a $K$-point
$e_0$ of $\Lb(Q/\Cc_K/K)$ and a closed subscheme $Z_{\geq e_0}$ of
$\Sigma^{\geq e_0}(Q/\Cc_K/K)$ satisfying the following.
%, after base change to make
%$K$ algebraically closed, all of the following hold.
\begin{enumerate}
\item[(i)]
For every geometric point $e$ of $\Lb^{\geq e_0}(Q/\Cc_K/K)$, the
closed subscheme $Z_e\subset \Sigma^e(Q/\Cc_K/K)$ is an irreducible
component parameterizing (among others) a section $\sigma$ whose
normal bundle has vanishing $h^1$ 
even after twisting down by divisors of degree
$2g(\Cc_K)+1$, i.e., the section is \emph{$(g)$-free}.
\item[(ii)] The restricted Abel map is dominant,
and
the geometric generic fiber is integral, resp.
``birationally rationally connected,''
$$
\alpha_{\mcT_K,Z_e}: Z_e\cap\Sec(\Xx_\FR/\Cc_\FR/\FR)    \to |BQ^e_{\Cc_K/K}|.
$$
\item[(iii)] For every $(g)$-free section $\sigma$, for all
integers  $\delta\geq \delta_0$, after attaching to $\text{Image}(\sigma)$ 
at $\delta \rho$ general points $\delta$ curves in fibers of $f_K$ for
each of the $\rho$ minimal, extremal curve
class in the Mori cone of $f_K$, the resulting reducible curve in $\Xx_K$ is
parameterized by $Z_e$ for some $e\in \Lb^{\geq
  e_0}(Q/\Cc_K/K)$.   
\end{enumerate}
A quasi-projective scheme over a characteristic $0$, algebraically
closed field is \emph{birationally rationally connected}
if it is integral and every projective model is rationally connected.
\end{defn}

\begin{hyp} \label{hyp-Abel} \marpar{hyp-Abel}
Assume \ref{hyp-Sec} and existence of an Abel
sequence.  
\end{hyp}

\begin{defn} \label{defn-main-b} \marpar{defn-main-b}
For an integral model that satisfies Hypotheses \ref{hyp-positiveMori}
and \ref{hyp-Abel}, for every integer $e\geq e_0$, define $Z_{\Rr,e}$
to be the closure in $\text{Hilb}_{\Xx_\Rr/\Rr}$ of $Z_e$.  
\end{defn}

\begin{thm} \label{thm-main-b} \marpar{thm-main-b}
For an integral model that satisfies Hypotheses \ref{hyp-positiveMori}
and \ref{hyp-Abel}, for every integer $h\geq |e_0|$, 
for every general
$K$-point $I$ of $|BQ^{h\theta}_{\Cc_\FR/\FR}|$, there exists an
$F$-point $[\mcC_F]$ of $Z_{\Rr,h\theta}$ giving
an $F(\eta)$-point of $\Xx_\eta$ and such that for a
finite, local homomorphism of DVRs, $R\to R'$, $\mcC_F\otimes_F F'$ is
the closed fiber of closed subscheme $\mcC_{R'}\subset \Xx_R\otimes_R
R'$ whose generic fiber is a section of $f$ with Abel image equal to
$I\otimes_\FR \FR'$. 
\end{thm}

\begin{proof}
By Hironaka's resolution of
singularities in characteristic $0$ and by generic smoothness in
characteristic $0$, there exists a closed subscheme $Y_{K,e}\subset
Z_{e}$ such that for the associated blowing up,
$\nu_{K,e}:\tilde{Z}_{e}\to Z_{e}$, the composition
$$
\alpha_{\mcT_K}|_{Z_{K,e}}\circ \nu_{K,e}:\tilde{Z}_{e} \to Z_e \dashrightarrow
|BQ^e_{\Cc_K/K}| 
$$
is everywhere regular, and this regular morphism is
smooth over a dense open subset $U_e$
of $|BQ_{\Cc_K/K}|$.  
%Moreover we may assume that $U_e$
%is contained in a specified dense open.
After further shrinking, we may assume that the inverse image in
$\tilde{Z}_{e}$ 
of
the dense open subset $\Sec^e(\Xx_K/\Cc_K/K)\cap Z_{e}$ has nonempty
intersection with the fiber of $\alpha_{\mcT_K}$ over every point of
$U_e$.  
Denote this composition by $\tilde{\alpha}_{\mcT_K}$, and denote by
$\tilde{Z}_{U_e}$ the inverse image of $U_e$ under this morphism,
$$
\tilde{\alpha}_{\mcT_K,U_e}:\tilde{Z}_{U_e}\to U_e.
$$  

\mni
By construction, 
$\tilde{\alpha}_{\mcT_K,U_e}$ is projective.  By construction,
$\tilde{\alpha}_{\mcT_K,U_e}$ is smooth.  Finally, since the
geometric generic fiber of $\alpha_{\mcT_K}|_{Z_{e}}$ is nonempty,
irreducible and birationally rationally connected, also the
geometric generic fiber of
$\tilde{\alpha}_{\mcT_K,U_e}$ is integral and rationally connected.  Since
$\tilde{\alpha}_{\mcT_K,U_e}$
is a smooth, projective morphism of
varieties over a characteristic $0$ field whose geometric generic
fiber is irreducible and rationally connected, in fact \emph{every}
geometric fiber is irreducible and rationally connected,
cf. \cite[2.4]{KMM}, \cite[Theorem IV.3.11]{K}.
By \cite[Corollary 1.2]{Esnault_integrality}, it follows that every
geometric fiber has coniveau $\geq 1$.

\mni
Denote by $Y_{\Rr,e}$ the closure of $Y_{e}$ in $Z_{\Rr,e}$, and denote by
$\nu_{\Rr,e}:\tilde{Z}_{\Rr,e} \to Z_{\Rr,e}$ the associated blowing up.
Let $I$ be a $K$-point of $U_e$.  Denote by $\tilde{Z}_{I}$ the
inverse image of $I$ under $\tilde{\alpha}_{\mcT_K}$, and
denote by $\tilde{Z}_{\Rr,I}$ the closure of
$\tilde{Z}_{K,I}$ in $\tilde{Z}_{\Rr,e}$.  
By the previous paragraph, $\tilde{Z}_{I}$ is a smooth, projective
$K$-scheme that is geometrically irreducible and has coniveau $\geq 1$.

\mni
As the closure of $\tilde{Z}_{I}$
in a projective ${\Rr}$-scheme,
also $\tilde{Z}_{\Rr,I}$ is a projective ${\Rr}$-scheme.  
Since 
$\tilde{Z}_{I}$ is
a projective $K$-scheme, also 
the $K$-fiber of the closure
$\tilde{Z}_{\Rr,I}$ equals $\tilde{Z}_{I}$.  Since ${\Rr}$
is a DVR, the ${\Rr}$-scheme $\tilde{Z}_{\Rr,I}$
is flat if and only if every associated point dominates the generic
point
$\SP(K)$ of $\SP(\Rr)$,
and this holds since $\tilde{Z}_{\Rr,I}$ equals the closure of its
generic fiber.  Thus $\tilde{Z}_{\Rr,I}$ is a proper, flat ${\Rr}$-scheme
whose geometric generic fiber is smooth, integral, and has coniveau
$\geq 1$.    
Thus, by \cite[Corollary 1.2]{Esnaultpadic}, the fiber
$\tilde{Z}_{F,I}$ over the finite residue field $F$ has an $F$-rational
point $[\mc{C}_F]$. 

\mni 
Finally, since $I$ is in $U_e$, the fiber
$\tilde{Z}_{I}$ is an integral scheme that intersects the
open set $\nu_{K,e}^{-1}(\Sec^e(\Xx_K/\Cc_K/K))$.  Thus this open set is
dense in $\tilde{Z}_{I}$, and hence also in
$\tilde{Z}_{\Rr,I}$.  So there exists a finite type,
flat, local homomorphism of DVRs, ${\Rr}\to
\Rr'$, whose field extensions $K\subset K'$ and $F\subset F'$ are
finite, and there exists an ${\Rr}$-morphism $e:\SP(\Rr')\to
\tilde{Z}_{\Rr,I}$ mapping $\SP(F')$ to the point $[\mcC_F]$
and mapping $\SP(K')$ into $\nu_{K,e}^{-1}(\Sec^e(\Xx_K/\Cc_K/K))$.  The
composition of $e$ with $\nu$ gives a deformation after base change
whose closed fiber is $\mc{C}_F\otimes_F F'$ and whose generic fiber
maps to $I$ under $\alpha_{\Ll}$.  

\mni
Finally, by the definition of $|e_0|$, for every integer $h\geq
|e_0|$, for $e=h\cdot m \cdot \theta$, 
$e$ is an element of $\Lb^{\geq e_0}(Q/\Cc_K/K)(K)$ such that
$|BQ^{e}_{\Cc_K/K}|$ has a $K$-point.
Since $K$ is the fraction field of a Henselian DVR, by Hensel's Lemma,
every smooth
$K$-variety that has a $K$-point has a Zariski dense collection of
$K$-points. 
In particular, the dense open subscheme $U_e$ of
$|BQ^e_{\Cc_K/K}|$ has a $K$-point, $I$.  By the above,
there exists an $F$-point 
$[\mc{C}_F]\in \Sigma^e(\Xx_{\Rr}/\Cc_{\Rr}/{\Rr})(F)$ and a
deformation after base change $(\Rr\to \Rr',\mc{C}_{\Rr'})$ such that
$\alpha_{\mcT}(\sigma_{K'})$ equals $I\otimes_K K'$.
Finally, by Lemma \ref{lem-Sigma}, there exists an open subscheme
$V_{\mc{C}_F}$ of $\Cc_F$ containing every generic point of
$\Cc_{F,\text{sm}}$ and over which $f|_{\mc{C}_F}:\mc{C}_F \to \Cc_F$ is an
isomorphism.  Inverting this isomorphism gives a section of $f$ over
$V_{\mc{C}_F}$, and hence gives an $F(\eta)$-rational point of
$\Xx_{F(\eta)}$ for every generic point $\eta$ of $\Cc_F$.  
\end{proof}

%%%%%%%%%%%%%%%%%%%%%%%%%%%%%%%%%%%%%%%%%%%%%%%%%%%%%%%%%%%%%%%
%%
%% Section: Integral Extensions
%% 
%%%%%%%%%%%%%%%%%%%%%%%%%%%%%%%%%%%%%%%%%%%%%%%%%%%%%%%%%%%%%%%

\section{Integral Extensions} \label{sec-extend}
\marpar{sec-extend}

\mni
For each of the four special cases, we explain how to find integral
extensions whose generic fiber is a rationally simply connected
fibration over a curve.  Combined with Theorem \ref{thm-main-b}, this
produces rational points.

%%%%%%%%%%%%%%%%%%%%%%%%%%%%%%%%%%%%%%%%%%%%%%%%%%%%%%%%%%%%%%%
%%
%% Subsection: Projective Homogeneous Spaces
%% 
%%%%%%%%%%%%%%%%%%%%%%%%%%%%%%%%%%%%%%%%%%%%%%%%%%%%%%%%%%%%%%%

\subsection{Projective Homogeneous Spaces} \label{subsec-PHS}
\marpar{subsec-PHS}

\begin{defn} \label{defn-gs} \marpar{defn-gs}
For a stack $\mcM$ over a scheme $S$,
a \emph{generic splitting variety} of $\mc{M}$ (ala Amitsur)
is a pair $(M,\zeta_M)$ of a
locally finitely presented $S$-scheme $M$ and a $1$-morphism over $S$,
$$
\zeta_M:M\to \mc{M},
$$ 
such that for
every field $E$ and for every $1$-morphism, $\zeta:\SP E \to
\mc{M}$, there exists a morphism $z:\SP E \to M$ with
$\zeta_M\circ z$ equivalent to $\zeta$.  
%(In the applications here, it
%suffices to factor $\zeta$ as $\zeta_M\circ z$ only for infinite
%fields $E$.) 
\end{defn}

\mni
For every fppf group $S$-scheme $G_S$, the \emph{classifying stack}
$BG_S$ is the stack over the fppf site of $S$-schemes whose fiber
category over an $S$-scheme $T$ has as objects the $G_S$-torsors
over $T$ and has as morphisms the $G_S$-equivariant $T$-morphisms of
torsors \cite[Proposition 10.13.1]{LM-B}.
A smooth, affine group scheme $G_{S,0}$ over $S$ is a \emph{reductive
  group scheme} if the geometric fibers are connected and
(geometrically) reductive, i.e., the unipotent radical is trivial.  A
flat, affine group scheme $G_S$ over $S$ is \emph{finite-by-reductive}
if there exists a surjective, smooth morphism of group $S$-schemes
$$
G_S \twoheadrightarrow \pi_0 G_S
$$
where $\pi_0 G_S$ is a finite, flat group $S$-scheme and the kernel,
$G_{S,0}$, is a reductive group $S$-scheme.

\begin{thm}\cite{dJS8} \cite[Section 16]{dJHS} \label{thm-SdJ}
  \marpar{thm-SdJ} 
Let $R$ be a DVR.
For every ``finite-by-reductive'' group $R$-scheme, $G_R$,  
for every integer
$c\geq 1$, there exists a projective, fppf $R$-scheme $\ol{M}$, an
open subscheme $M$ whose closed complement $\partial \ol{M}$ has
codimension $>c$ in every fiber $\ol{M}_s$,
%projective datum with codimension
%$>c$ boundary, 
%$(\ol{M}\to S,\OO_{\ol{M}}(1),i:M\to\ol{M})$, 
and a
$1$-morphism $\zeta_M:M\to BG_S$ that is a generic splitting variety.
\end{thm}

\begin{proof}
There exists a linear representation of $G_R$ on a finite free
$R$-module $V$ whose induced action on $\PP_R V$ is free on a dense
open subset $(\PP_R V)^o$ whose complement in the semistable locus
$(\PP_R V)^{\text{ss}}$ has codimension $>c$ in every fiber over $\SP R$.  By
\cite{Sesh}, the uniform categorical quotient of the action of $G_R$
on $(\PP_R V)^{\text{ss}}$ is a projective $R$-scheme $\ol{M}$, and
there exists a unique open subscheme $M$ such that $(\PP_R V)^o$ is a
$G_R$-torsor over $M$.  This defines $\zeta_M:M\to BG_R$.  

\mni
By construction, $\ol{M}$ is normal and integral.  Since it is also
projective over the DVR $R$, and since it dominates the generic point
$\SP \FR$, $\ol{M}$ is $R$-flat.  Since the complement of 
$(\PP_R V)^o$ has codimension $>c$, also $\partial \ol{M}$ has
codimension $>c$.  Since $(\PP_R V)^\circ$ is $R$-smooth, and since
$(\PP_R V)^o\to M$ is flat, also $M$ is $R$-smooth, 
\cite[Proposition 17.7.7]{EGA4}.

\mni
For every $\zeta:\SP E \to BG_R$, the $2$-fibered product
$M\times_{\zeta_M,\zeta} \SP E$ with its projection to $\SP E$ is
is $(\PP_E V_\zeta)^o$ where
$V_\zeta$ is the $E$-vector space 
$\text{Hom}_E(E^{\oplus c+1},\text{Hom}_E(W_0,W))$ for finite free
$E$-vector spaces $W_0$, $W$ of equal (positive) rank.  The free locus
contains the locus parameterizing $(c+1)$-tuples
$(\lambda_0,\dots,\lambda_c)$ 
of
$E$-linear maps $\lambda_i:W_0\to W$ that are isomorphisms.  This locus has
$E$-rational points.  Thus, there exists an $R$-morphism $\lambda:\SP
E \to M$ and a $2$-equivalence of the composition $\zeta_M\circ
\lambda$ with $\zeta$.  
\end{proof}

\mni
Let $H_S$ be a split, connected and simply connected, semisimple group
scheme over  
$S$.  Let $P_S\hookrightarrow H_S$ be a standard parabolic subgroup
scheme (containing a specified Borel).  This has a Levi decomposition.
Denote by $\chi:P_S\to Q_S$ the multiplicative quotient, i.e., the
quotient of $P_S/R_u(P_S)$ by the commutator subgroup
$S$-scheme.  Consider the left regular action of $P_S$ on $H_S$,
$p\cdot h = ph$.  Consider the diagonal left action of $P_S$ on
$Q_S\times_S H_S$, $p\cdot (q,h) = (q\chi(p)^{-1},ph)$.  The
projection morphism $\pr{2}:Q_S\times_S H_S \to H_S$ is
$P_S$-equivariant.  Finally, the left regular $Q_S$-action,
$$
Q_S \times_S (Q_S\times_S H_S) \to Q_S\times_S H_S, \ q'\cdot (q,h) = (q'q,h),
$$
is $P_S$-equivariant.  
By fppf descent, there are affine morphisms
$$
H_S \to P_S\backslash H_S, \ \ Q_S\times_S H_S \to P_S\backslash
(Q_S\times_S H_S)
$$
that are $P_S$-torsors.  
The $P_S$-equivariant morphisms above
induce morphisms
$$
\pr{2}:P_S\backslash (Q_S\times_S H_S) \to P_S\backslash H_S,
$$
$$
Q_S\times_S P_S\backslash (Q_S\times_S H_S) \to P_S\backslash (Q_S\times_S H_S).
$$
Moreover, the scheme
$P_S\backslash H_S$ is projective over $S$: the dual of the relative
dualizing sheaf is relatively very ample, cf. \cite[p. 186]{Demazure-Aut}.
%One proof uses \cite[Corollaire 8.1.1,
%Proposition 10.13.1]{LM-B} to construct $P_S\backslash H_S$ as a
%separated, smooth, quasi-compact
%algebraic
%space over $S$ such that $H_S$ is a $P_S$-torsor over this algebraic space.
%In fact, using cohomology and base change and Bott vanishing, \cite{DemBott},  
%the dual invertible sheaf
%of the relative dualizing sheaf over $S$ is $S$-very ample, has
%(universally) vanishing higher direct images, and gives a closed
%immersion into a projective space bundle.  
The other schemes are
affine over $P_S\backslash H_S$, hence can be constructed by fpqc
descent for affine schemes.
%, cf. Theorem \ref{thm-descent}. 
%Theorem \ref{thm-descent} and the proof of
%Corollary \ref{cor-psi}.  
Altogether these quotients and 
morphisms make $\mcT_S := P_S\backslash (Q_S\times_S
H_S)$ into a 
(left) $Q_S$-torsor over $P_S\backslash H_S$.  

\begin{defn} \label{defn-PHS} \marpar{defn-PHS}
For the pair
$(H_S,P_S)$ as above, the associated \emph{projective homogeneous scheme} is
$X_{H,P} = P_S\backslash H_S.$  
\end{defn}

\begin{lem}\cite[Construction 5.3]{Zhu} \label{lem-UT} \marpar{lem-UT}
The Picard torus of $X_{H,P}$ over $S$ is $Q_S$ and the Picard lattice
is the character lattice $Q_S^D$.
The $Q_S$-torsor $\mcT_S$ is a
universal torsor on $P_S\backslash H_S$.
\end{lem}

\mni
Since $H_S\to X_{H,P}$ is a $P_S$-torus, since $P_S$ is $S$-flat, and
since $H_S$ is $S$-smooth, $X_{H,P}$ is $S$-smooth, 
\cite[Proposition 17.7.7]{EGA4}.
%For the relative
%dualizing sheaf $\omega_{X/S}$, the dual invertible sheaf
%$\omega_{X/S}^\vee$ is very ample and has vanishing higher direct
%image sheaves, compatibly with arbitrary base change of $S$.  
%Since $X_{H,P}$ is flat and proper over $S$, and using cohomology and
%base change, \cite[Theorem III.12.11]{H},
%this can be checked
%on fibers over geometric points of $S$.  Over fields, this follows by
%Bott vanishing, \cite{DemBott}. 

\begin{notat} \label{notat-rho} \marpar{notat-rho}
Denote by $\rho_X$ the unique morphism of group $S$-schemes, 
$$
\rho_X:Q_S \to \Gm{m,S},
$$
such that the associated $\Gm{m,S}$-torsor $\rho_{X,*}\mcT_S$ is the
$\Gm{m,S}$-torsor of the very ample invertible sheaf that is the dual
of the relative dualizing sheaf.  
\end{notat}
     
\mni
The automorphism group scheme of the pair
$(X_{H,P},\omega_{X/S}^\vee)$ is an affine group $S$-scheme, cf.
\cite[Section 2.1]{dJS8}.  This is  
essentially just \cite[No. 221, Section 4.c]{FGA}, \cite[Corollaire
7.7.8]{EGA3}, \cite[Th\'{e}or\`{e}me 4.6.2.1]{LM-B}.  The same method
proves that the 
automorphism group scheme $G'_S$ of the pair $(X_{H,P},\mcT_S)$ is an affine
group $S$-scheme.  Since $\mcT_S$ is intrinsic, the forgetful morphism
$$
\text{Aut}_S(X_{H,P},\mcT_S) \to \text{Aut}_S(X_{H,P})
$$
is surjective.  This group scheme is smooth, and the 
identity component of the target is a semisimple
group $S$-scheme, \cite[Proposition 4]{Demazure-Aut}.  The kernel of
the forgetful morphism is
$Q_S$ with its induced action on $\mcT_S$.  Thus, $G'_S$ is a smooth,
affine group $S$-scheme whose identity component $G_{S,0}$ is a
reductive group scheme.

\mni
The quotient $G'_S/G_{S,0}$ is a quasi-finite, \'{e}tale group scheme
$\pi_0 G'_S$ over $S$.  
When $S$ is $\SP R$ for a Henselian local ring, there is a
unique closed subgroup scheme $\pi_0 G'_{S,\text{fin}}$ of 
$\pi_0 G'_S$ that is finite,
\'{e}tale over $S$ and whose closed fiber $\pi_0 G'_{F,\text{fin}}$ equals the
closed fiber $\pi_0 G'_F$.  

%\begin{defn} \label{defn-G} \marpar{defn-G}
%For $S$ equal to $\SP R$ for a Henselian local ring $R$, associated to
%the pair $(H_S,P_S)$, define $G_S$ to be the closed subgroup scheme of
%$G'_S$ that contains $G_{S,0}$ and with $\pi_0 G_S$ as defined above.
%\end{defn}
 
\mni
By construction $G_S$ is finite-by-reductive.
Denote the induced action of $G'_S$, as a
\emph{right} action
$$
\gamma_X: X_{H,P} \times_S G'_S \to X_{H,P}, \ \
(P_S\cdot h,g) \mapsto (P_S\cdot h)\bullet g.
$$
$$
\gamma_\mcT: \mcT_S \times_S G'_S \to \mcT_S.
$$
Let $G_S\subset G'_S$ be a
closed subgroup scheme that contains $G_{S,0}$ and that is
finite-by-reductive, i.e., $G_{S,0}$ is the kernel of a smooth,
surjective group homomorphism $G_S\to \pi_0 G_S$ with $\pi_0 G_S$
finite and flat.  
The action of $G_S$ on $X_{H,P}$ induces
an action on the Picard lattice $Q_S^D$.  This action factors through
an action of $\pi_0 G_S$ on $Q_S^D$.  

\begin{defn} \label{defn-prim} \marpar{defn-prim}
The triple $(H_S,P_S,G_S)$ of a split, simply connected, semisimple
group $S$-scheme $H_S$, a standard parabolic $P_S$, and a finite-by-reductive
subgroup $S$-scheme of $\text{Aut}_S(X_{H,P},\mcT_S)$
containing the identity component
is \emph{primitive} if the induced action
of $\pi_0 G_S$ on the Picard lattice $Q^D_S$ of $P_S\backslash H_S$ 
has a rank $1$ invariant
sublattice.  
\end{defn}

\mni
Since it is intrinsic, one nonzero element in the invariant lattice is the
character $\rho_X$ giving rise to the ample invertible sheaf
$\omega^\vee_{X/S}$.   Thus, the triple is primitive if and only if
every invariant character is commensurate with $\rho_X$.

\mni 
For every $S$-scheme $T$ and for every left $G_S$-torsor $E_T$ over $T$, 
$$
(\mu,\pr{2}):G_S\times_S E_T \xrightarrow{\cong} E_T\times_T E_T, \ \
(g,x)\mapsto 
(\mu(g,x),x), 
$$
there are induced left actions of $G_S$ as follows,
$$
\tilde{\mu}_X:G_S \times_S (X_{H,P}\times_S E_T) \to
X_{H,P}\times_S E_T, \ \ (g,(x,y)) \mapsto (x\gamma_X(g^{-1}),\mu(g)y), 
$$
$$
\tilde{\mu}_\mcT:G_S \times_S \mcT_S\times_S E_T \to
\mcT_S\times_S E_T, \ \ (g,(q,y))\mapsto (q\gamma_\mcT(g^{-1}),\mu(g)y).
$$
These are free action.  Associated to the $G_S$-invariant cocharacter 
$\rho_X$, the action
$\tilde{\mu}_\mcT$ induces a $G_S$-linearization over $X_{H,P}\times_S
E_T$ of the relatively ample invertible sheaf
$\pr{1}^*\omega_{X/S}^\vee$.  Thus, 
by
fppf descent in the projective case, 
%(essentially fppf descent in
%the affine case as in Theorem \ref{thm-descent} for 
%the associated graded algebra of $\omega_{X/S}^\vee$), 
there is
an affine $T$-morphism that is a $G_S$-torsor,
$$
q_{\Xx,E}: \Xx_{H,P}\times_S E_T \to \Xx_E,
$$ 
together with a $G_S$-torsor,
$$
q_{\mcT,E}:\mcT_S \times_S E_T \to \mcT_E,
$$
and a quotient morphism 
$$
\mcT_E \to \Xx_E
$$
that is a $Q_S$-torsor.

\begin{defn} \label{defn-twist} \marpar{defn-twist}
The pair $(\Xx_E,\mcT_E)$ is the \emph{$E$-twist} of $(\Xx_{H,P},\mcT_S)$.  
\end{defn}

\begin{notat} \label{notat-prim} \marpar{notat-prim}
%Let $(H_S,P_S,G_S)$ be primitive.  
For a triple $(H_S,P_S,G_S)$,
let $(\ol{M},\OO_{\ol{M}}(1),
i:M\to \ol{M})$ be a datum with codimension $>1$
boundary as in Theorem \ref{thm-SdJ}.  For the $1$-morphism
$\zeta_M:M\to BG_S$,
denote by $f_M:\Xx_M \to M$ and $\mcT_M$
the associated twist of
$(\Xx_{H,P},\mcT_S)$.  
%For every integer $r\geq 1$, denote by
%$W_r\subset G_r$ the dense open from Lemma \ref{lem-Bertini}.
\end{notat}

\begin{hyp} \label{hyp-group} \marpar{hyp-group}
Let $S$ be $\SP R$ for a Henselian DVR with finite residue field $F$ and
with characteristic $0$ fraction field $\FR$.
Let $F(\eta)$ be $F(\Cc'_F)$ for a smooth, projective,
geometrically connected $F$-curve $\Cc'_F$.  Let $X_\eta$ be an
$F(\eta)$-scheme that is projective homogeneous.  Assume
that $X_\eta$ has a universal torsor $\mcT_\eta$.
\end{hyp}

\mni
There exists a pair
$(H_S,P_S)$ as above such that $X_{H,P}\otimes_R F(\eta)$ is
geometrically isomorphic to $X_\eta$.  Let $Q_\eta$ be the Picard
torus of $X_\eta$.

\begin{notat} \label{notat-group} \marpar{notat-group}
Denote by $G_S\subset
\text{Aut}_S(X_{H,P},\mcT_{H,P})$ the unique closed subgroup scheme
that is finite-by-reductive, that contains the identity component, and
such that $\pi_0 G_F$ is the group of connected components of
$\text{Aut}_F(X_{H,P},\mcT_{H,P})$.  
\end{notat}

\mni
The Isom scheme from $(X_\eta,\mcT_\eta)$ to $(X_{H,P},\mcT_{H,P})$ is a
left $G_S$-torsor over $F(\eta)$.  The $S$-smooth scheme $M$ from
Theorem \ref{thm-SdJ}, for $c=1$, 
is a generic splitting variety.  Thus, there is
an induced morphism $\zeta:\Cc'_F\to \ol{M}$.  If necessary, replace
$\ol{M}$ by $\ol{M}\times_S \PP^3_S$ so that $\zeta$ is a closed
immersion.  Denote the relative dimension of $\ol{M}$ by $m$.  Denote
by $(\Xx_M,\mcT_M)$ the twist of $(X_{H,P},\mcT_{H,P})$ by the torsor
$\zeta_M:M\to BG_S$.  

\begin{thm}\cite{Zhu} \label{thm-group} \marpar{thm-group}
Assume Hypothesis \ref{hyp-group}.
For every general $(m-1)$-tuple $(D_1,\dots,D_{m-1})$ of very ample
divisors of $\ol{M}$ that contain $\zeta(\Cc'_F)$, the common
intersection $\Cc_R = D_1\cap \dots \cap D_{m-1}$ gives an integral
extension of $X_\eta \to \SP F(\eta)$.  If $(H_S,P_S,G_S)$ is
primitive, then $\Xx_\FR \to \Cc_\FR$ and $\mcT_\FR$ is a rationally
simply connected fibration in the sense of \cite{Zhu}, and the family
has an Abel sequence. 
\end{thm}

\begin{proof}
Since $\Cc'_F$ is smooth, and since the image of $\zeta$ intersects
the smooth open subscheme $M$, there exist very ample divisors
$(D_{1,F},\dots, D_{m-1,F})$ that contain $\zeta(\Cc'_F)$ and whose
common intersection $\Cc_F$ is a curve containing a dense open subset
of $\zeta(\Cc'_F)$ as an open subset.  Since the boundary
$\partial\ol{M}$ has codimension $>1$ in $\ol{M}$, for a general lift
$(D_1,\dots,D_{m-1})$ (inside the complete linear system $\PP^N_\Rr$
over the Henselian ring $\Rr$), the curve $\Cc_\FR$ is a smooth curve
that is contained in $M$.  Thus, the restriction of $(\Xx_M,\mcT_M)$
gives an integral extension of $(\Xx_\eta,\mcT_\eta)$.

\mni
By the Bertini irreducibility theorem,
for a general complete
intersection curve $\Cc_\FR$, the sections of the relative Picard of
$\Xx_M/M$ over $\Cc_\FR$ equal the sections over all of $M_\FR$.  If
$(H_S,P_S,G_S)$ is primitive, the the group of sections is a free
Abelian group of rank $1$.  Thus, the restricted family $\Xx_\FR\to
\Cc_\FR$ with the restricted universal torsor $\mcT_\eta$ satisfies
\cite[Situation 5.1]{Zhu}. By the proof of \cite[Theorem 1.4]{Zhu},
this family has an Abel sequence.  
\end{proof}

\begin{cor}\cite[Theorem 1.4, Lemma 12.1]{Zhu} \label{cor-group}
  \marpar{cor-group} 
Assume Hypothesis \ref{hyp-group}.  There exists an $F(\eta)$-point of
$\Xx_\eta$.  
\end{cor}

\begin{proof}
This is proved by induction on the dimension of $\Xx_\eta$.  If the
triple $(H_S,P_S,G_S)$ is primitive, then there exists an
$F(\eta)$-point by Theorem \ref{thm-main-b} and Theorem
\ref{thm-group}.  If the triple is not primitive, then there exists
$G_S$-stabilized standard parabolic subgroup $P'_S$ that strictly contains
$P_S$, cf. \cite[Lemma 12.1]{Zhu}.  Thus $X_{H,P'}$ has positive
dimension that is strictly
smaller than the
dimension of $X_{H,P}$.  Denote by $\Xx'_M$ the 
twist of $\Xx_{H,P'}$ over $M$.  The $(H_S,G_S)$-equivariant morphism
$X_{H,P}\to X_{H,P'}$ gives a 
smooth, projective morphism $\Xx_M\to \Xx'_M$.  By induction on the
dimension, $\Xx'_\eta = \zeta^*\Xx'_M$ has an $F(\eta)$-point.  The
fiber of $\Xx_\eta \to \Xx'_\eta$ over this $F(\eta)$-point is a
standard projective homogeneous space for the Levi factor of $P'_S$.
By the induction hypothesis, this projective homogeneous space has an
$F(\eta)$-point.   
\end{proof}

%%%%%%%%%%%%%%%%%%%%%%%%%%%%%%%%%%%%%%%%%%%%%%%%%%%%%%%%%%%%%%%
%%
%% Subsection: Complete Intersections
%% 
%%%%%%%%%%%%%%%%%%%%%%%%%%%%%%%%%%%%%%%%%%%%%%%%%%%%%%%%%%%%%%%

\subsection{Complete Intersections} \label{subsec-CI}
\marpar{subsec-CI}

\begin{notat} \label{notat-CI-b} \marpar{notat-CI-b}
Let $S$ be a scheme.
Let $n$ and $1\leq b \leq n$ be positive integers.  Let
$\ul{d}=(d_1,\dots,d_b)$ be an ordered $b$-tuple of integers $d_i\geq
2$.  
For each
$j=1,\dots,\cb$, denote by $V_j(d_j)$ the free $\OO_S$-module
$H^0(\PP^{n}_S,\OO_{\PP^{r_j}_S}(d_j))$.
Denote by
$V(\underline{d})$ the direct sum $V_1(d_1)\oplus \dots \oplus
V_\cb(d_\cb)$ as a free $\OO_T$-module.  Denote by $\PP_S
V(\underline{d})$ the projective space over $S$ on which there is a
universal ordered
$\cb$-tuple $(\phi_1,\dots,\phi_\cb)$ 
of sections of the invertible sheaves $\OO_{\PP^{n}}(d_j)$. 
Precisely, for
the product 
$$
P=\PP_S V(\underline{d}) \times_S \PP^n_S
$$
with its projections 
$$
\text{pr}_1:P\to
\PP_TV(\underline{d}) \text{ and }
\text{pr}_2:P\to \PP_T^{n},
$$
the sequence
$(\phi_1,\dots,\phi_\cb)$ is a universal homomorphism of coherent sheaves
$$
\text{pr}_2^*\OO_{\PP^{n}_S}(-d_1)\oplus \dots \oplus
\text{pr}_2^*\OO_{\PP^{n}_S}(-d_\cb) \to \text{pr}_1^*\OO_{\PP_S
  V(\underline{d})}(1),  
$$
or equivalently, a universal homomorphism of coherent sheaves,
$$
(\phi_1,\dots,\phi_\cb):\text{pr}_1^*\OO_{\PP_S
  V(\underline{d})}(-1)\otimes\left(\text{pr}_2^*\OO_{\PP^{n}_T}(-d_1)\oplus
  \dots \oplus 
\text{pr}_2^*\OO_{\PP^{n}_T}(-d_\cb)\right) \to \OO_P.  
$$
For each $j=1,\dots,b$, denote by $Y_j\subset P$ the effective Cartier
divisor defined by $\phi_j$.  Denote by $X_j$ the intersection
$Y_1\cap \dots \cap Y_j$ as a closed subscheme of $P$.
\end{notat}

\begin{defn} \label{defn-CI2-b} \marpar{defn-CI2-b}
For every $j=1,\dots,\cb$,
the \emph{smooth locus} $\PP_S V(\underline{d})^{\text{sm}}_j$, resp. the
\emph{Lefschetz locus} $\PP_S V(\underline{d})^{\text{Lef}}_j$, is
the maximal open subscheme of $\PP_S V(\underline{d})$ over which
$\Xx_j$ is flat and every geometric fiber is smooth, resp. has at
worst a
single ordinary double point.  The \emph{degenerate locus}, resp. the
\emph{badly degenerate locus}, is the closed complement of $\PP_S
V(\underline{d})_j^{\text{sm}}$, resp. $\PP_S V(\ul{d})_j^{\text{Lef}}$.
\end{defn}

\begin{prop} \cite[Expos\'{e} XVII, Th\'{e}or\`{e}me
    2.5]{SGA7II} \label{prop-Lef-b} \marpar{prop-Lef-b} 
Assume that $\cb < n$.  Then for every $j=1,\dots,\cb$,
the degenerate locus is a proper closed
subset, and the badly degenerate locus of $\Xx_j$ is a proper closed subset of
codimension $>1$.
\end{prop}

\begin{proof} 
This is proved by induction on $j$.  For $j=1$, this follows from
loc. cit.  By way of induction, assume that $j>1$ and assume that the
result is proved for $j-1$.  By loc. cit., the intersection with
$\PP_S V(\ul{d})^{\text{sm}}_{j-1}$ of the degenerate locus,
resp. badly degenerate locus, of $\Xx_j$ is a proper closed subset,
resp. has codimension $>1$.  It suffices to prove that
$\PP_S V(\ul{d})^{\text{sm}}_j$ contains every generic point $\xi$ of
the degenerate 
locus of $\Xx_{j-1}$ that is not in the badly degenerate locus.  By
hypothesis, $\Xx_{j-1,\xi}$ has a unique ordinary double point.  In
$\PP_S V(\ul{d})$, it is a codimension $1$ condition for $\phi_j$ to
vanish at this point.  On the complement of this proper closed subset,
the degenerate locus of $\Xx_j$ is a proper closed subset 
by Bertini's smoothness
theorem \cite[Th\'{e}or\`{e}me 6.3(2)]{Jou}.
\end{proof}

\begin{hyp} \label{hyp-CI-b} \marpar{hyp-CI-b}
Let $S$ be $\SP R$ for a Henselian DVR with finite residue field $F$ and
with characteristic $0$ fraction field $\FR$.
Let $F(\eta)$ be $F(\Cc'_F)$ for a smooth, projective,
geometrically connected $F$-curve $\Cc'_F$.  
Notations are as in \ref{notat-CI-b}. Assume that $\cb<n$.
Let $\Xx_\eta \subset
\PP^n_{F(\eta)}$ be an intersection $Y_{\eta,1}\cap \dots
\cap Y_{\eta,\cb}$ of hypersurface
$Y_{\eta,i}=\text{Zero}(\phi_{\eta,i})$ 
of degree $d_i$.  Denote by $\zeta:\Cc'_F\to \PP_R V(\ul{d})$ the
$R$-morphism of $(\phi_{\eta,1},\dots,\phi_{\eta,\cb})$.  Denote by
$m$ the relative dimension of $\PP_R V(\ul{d})$ over $\SP R$.
\end{hyp}

\begin{thm}\cite{DeLand} \label{thm-TsenLang} \marpar{thm-TsenLang}
Assume Hypothesis \ref{hyp-CI-b}.  
For every general $(m-1)$-tuple $(D_1,\dots,D_{m-1})$ of very ample
divisors of $\PP_R V(\ul{d})$ that contain $\zeta(\Cc'_F)$, the common
intersection $\Cc_R = D_1\cap \dots \cap D_{m-1}$ and the restriction
$\Xx_{R,\cb} \to \Cc_R$ of $\Xx_\cb$
gives an integral
extension of $X_\eta \to \SP F(\eta)$.  The restriction of the
$\Gm{m}$-torsor of $\OO(1)$ on $\PP^n_R$ is a universal torsor.
If
$d_1^2 + \dots + d_\cb^2 \leq n$, then $\Xx_{\FR,\cb}\to \Cc_{\FR}$ is
a rationally simply connected fibration (of Picard rank $1$) 
in the sense of \cite[Theorem 13.1]{dJHS},
and the family has an Abel sequence.
\end{thm}

\begin{proof}
Existence of the integral extension is basically the same as in the
proof of Theorem \ref{thm-group}.
By Proposition \ref{prop-Lef-b}, for a general choice of
$(D_1,\dots,D_{m-1})$, the curve $\Cc_\FR$ is a smooth, projective,
geometrically connected curve contained in the
Lefschetz locus of $\Xx_\cb$ and having dense intersection with the
smooth locus.  

\mni
Now assume that $d_1^2 + \dots + d_\cb^2\leq n$.
There are three global hypotheses in \cite[Theorem 13.1]{dJHS}, and
the remaining hypotheses are on the geometric generic fiber of
$f_\FR:\Xx_{\FR,\cb} \to \Cc_\FR$.  The first hypothesis is that $\Xx_{\FR,\cb}$
is smooth.  The projection $\Xx_\cb \to \PP^n_\FR$ is a projective
space bundle, hence $\Xx_\cb$ is smooth.  By Bertini's theorem,
\cite[Th\'{e}or\`{e}me 6.3(2)]{Jou}, for $D_1,\dots,D_{m-1}$ general,
the inverse image of $\Cc_\FR = D_{\FR,1}\cap \dots \cap D_{\FR,m-1}$
under the projection $\Xx_\cb \to \PP_R V(\ul{d})$ is smooth.  

\mni
The second hypothesis is that every geometric fiber of $f_\FR$ is irreducible.
By hypothesis, every $d_i\geq 2$, so that $d_1^2 + \dots + d_\cb^2
\geq 4\cb$.  Thus, $n\geq 4\cb$, so that also $n-\cb \geq 3\cb \geq
3$.  Every fiber of $f_\FR$ is a complete intersection of ample
divisors of dimension $\geq 3$.  Thus, by the lemma of
Enriques-Severi-Zariski, the complete intersection is connected.
Moreover, the complete intersection has at most a single ordinary
double point.  Thus, the fiber is normal, hence it is irreducible.

\mni
The third global hypothesis is that $\OO(1)$ is $f_\FR$-ample.  In
fact it is $f_\FR$-very ample since $\Xx_\FR$ is a closed subscheme of
$\Cc_\FR \times_{\SP \FR} \PP^n_\FR$.  

\mni
The remaining hypotheses are all hypotheses of the geometric generic
fiber of $f_\FR$, which is a smooth complete intersection with $d_1^2 +
\dots + d_\cb^2 \leq n$ that is general.  These
hypotheses are all established by Matt DeLand, \cite{DeLand}.
Therefore, by \cite[Theorem 13.1]{dJHS}, there exists an Abel sequence. 
\end{proof}

%%%%%%%%%%%%%%%%%%%%%%%%%%%%%%%%%%%%%%%%%%%%%%%%%%%%%%%%%%%%%%%
%%
%% Subsection: Hypersurfaces in Grassmannians
%% 
%%%%%%%%%%%%%%%%%%%%%%%%%%%%%%%%%%%%%%%%%%%%%%%%%%%%%%%%%%%%%%%

\subsection{Hypersurfaces in Grassmannians} \label{subsec-Grass}
\marpar{subsec-Grass}

\begin{notat} \label{notat-CI-b} \marpar{notat-CI-b}
Let $S$ be an affine scheme.  Let $m>0$ be an integer, and let $\ell>0$ be an
integer such that $m\geq 2\ell$ (or else replace $\ell$ by $m-\ell$).
Let $d>0$ be an integer.
Let $H_S$ be $\textbf{SL}_{m,S}$.  Let $P_S\subset H_S$ be the
parabolic that preserves the projection $\OO_S^{\oplus m} \to
\OO_S^{\oplus \ell}$ onto the first $\ell$ factors.  Thus,
$X_{H,P}=P_S\backslash H_S$ is the Grassmannian
$\text{Grass}_S(\ell,\mathcal{O}_S^{\oplus m})$.  Also the universal
torsor $\mcT_{H,P}$
is the $\Gm{m}$-torsor associated to the Pl\"{u}cker
invertible sheaf $\OO(1)$.  Let $G_S$ be $\text{Aut}_S(X_{H,P},\mcT_{H,P}).$
Let $W(d)$ be the
free $\OO_S$-module $H^0(X_{H,P},\OO(d)).$  Denote by $\PP_S
W(d)^{\text{Lef}}$ the Lefschetz locus parameterizing degree $d$
hypersurfaces that have at worst a single ordinary double point,
cf. Definition \ref{defn-CI2-b}.
\end{notat}

\mni
Let $V_S$ be the linear representation of $G_S$ from the proof of
Theorem \ref{thm-SdJ} with $c=1$.  
Define $\ol{M}'$ to be the uniform categorical
quotient of the action of $G_S$ on the semistable locus of
$\PP_S(V_S)\times_S \PP_S W(d)$.  By the proofs of Theorem
\ref{thm-SdJ} and Proposition \ref{prop-Lef-b}, 
there is an open
subscheme $M'\subset \ol{M}'$ whose inverse image equals $\PP_S(V_S)^o
\times_S \PP_S W(d)^{\text{Lef}}$ and that has closed complement
$\partial \ol{M}'$ of codimension $>1$.  Since $\PP_S(V_s)^o\times_S
\PP_S W(d)^{\text{Lef}} \to M'$ is flat, also $M'$ is $R$-smooth, 
\cite[Proposition 17.7.7]{EGA4}.  Denote by $G_{M'}\to M'$ and
$\mcT_{M'}$ 
the twist
of $\Xx_{H,P}$ and $\mcT_{H,P}$.  Denote by $\Xx_{M'}\subset G_{M'}$
the universal degree $d$ hypersurface.

\begin{hyp} \label{hyp-Grass} \marpar{hyp-Grass}
Let $S$ be $\SP R$ for a Henselian DVR with finite residue field $F$ and
with characteristic $0$ fraction field $\FR$.
Let $F(\eta)$ be $F(\Cc'_F)$ for a smooth, projective,
geometrically connected $F$-curve $\Cc'_F$.  Let $G_\eta$ be a smooth,
projective $F(\eta)$-scheme that is geometrically isomorphic to
$X_{H,P}$.  Assume that there exists a universal torsor $\mcT_\eta$.
Let $X_\eta\subset G_\eta$ be a closed subscheme that is geometrically
isomorphic to a degree $d$ hypersurface in $X_{H,P}$.  Denote by
$\zeta:\Cc'_F \to \ol{M}'_F$ the associated $R$-morphism.  Denote by
$s$ the relative dimension of $\ol{M}'$ over $\SP R$.
\end{hyp}

\begin{thm}\cite{Findley} \label{thm-Grass} \marpar{thm-Grass}
Assume Hypothesis \ref{hyp-Grass}.  For every general $(s-1)$-tuple
$(D_1,\dots,D_{s-1})$ of very ample divisors of $\ol{M}'_R$ that
contain $\zeta(\Cc'_F)$, 
the common
intersection $\Cc_R = D_1\cap \dots \cap D_{m-1}$
and the restriction
$\Xx_{R} \to \Cc_R$ of $\Xx_{M'}$
gives an integral
extension of $X_\eta \to \SP F(\eta)$.  The restriction of $\mcT_{M'}$
is a universal torsor.
If
$(3\ell -1)d^2 -d < m-4\ell-1$, then 
$\Xx_{\FR,\cb}\to \Cc_{\FR}$ is
a rationally simply connected fibration (of Picard rank $1$) 
in the sense of \cite[Theorem 13.1]{dJHS},
and the family has an Abel sequence.
\end{thm}

\begin{proof}
The first part of the proof is precisely the same as in the proof of
Theorem \ref{thm-TsenLang}.  The hypotheses on the geometric generic
fiber of $f_\FR$ follow from \cite{Findley}.
\end{proof}

\begin{proof}[Proof of Theorem \ref{thm-app}]

The theorem in Case 1 follows from 
Theorem \ref{thm-main-b} and 
Theorem \ref{thm-TsenLang}.  The theorem in
Case 3 follows from 
Theorem \ref{thm-main-b},
Theorem \ref{thm-group}, and Corollary \ref{cor-group}.
Case 2 is a special case of Case 3.  The theorem in Case 4 follows
from
Theorem \ref{thm-main-b} and 
Theorem \ref{thm-Grass}.
\end{proof}

\section{Acknowledgments} \label{sec-ack}
\marpar{sec-ack}

\mni
We are very grateful to Aise Johan de
Jong, in particular. 
%pointed us to results in \cite{EGA4}
%that are instrumental in Subsection \ref{subsec-S1}.  
We are grateful
to both de Jong and Eduardo Esteves who pointed out that an Abel map
as in \cite{dJHS} must exist under weaker hypotheses. 
The authors thank Yi Zhu for helpful
discussions and Max Lieblich who explained the history of the
Brauer-Hasse-Noether theorem.  We are grateful to Jean-Louis
Colliot-Th\'{e}l\`{e}ne and the referees for their feedback on an
early draft.
JMS was supported
by NSF Grants DMS-0846972 and DMS-1405709, as well as a Simons
Foundation Fellowship. 
CX was supported by the 
National Science Fund for Distinguished Young Scholars (11425101),
'Algebraic Geometry'. 

\bibliography{my}
\bibliographystyle{alpha}

\end{document}